\documentclass[letterpaper, 10 pt, conference]{ieeeconf}  

\IEEEoverridecommandlockouts
\usepackage{graphicx}
\usepackage[labelformat=simple]{subcaption}

\usepackage{xcolor}
\usepackage{amsmath} 
\usepackage{mathrsfs}
\usepackage{amssymb}
\usepackage{mathtools}
\usepackage{mathdots}
\usepackage{cite}
\usepackage{hyperref}
\let\labelindent\relax
\usepackage[shortlabels]{enumitem}
\usepackage{algorithm}
\usepackage[noend]{algpseudocode}
\usepackage{bm}
\usepackage{url}
\usepackage{optidef}
\usepackage{multirow}
\usepackage{balance}
\newtheorem{theorem}{\bf Theorem}[section]

\newtheorem{lemma}{\bf Lemma}[section]
\newtheorem{corollary}{\bf Corollary}[section]

\newtheorem{definition}{\bf Definition}[section]

\newcommand{\longthmtitle}[1]{\mbox{}\textup{\textbf{(#1):}}}

\newcommand{\real}{{\mathbb{R}}}

\newcommand{\integersnonnegative}{\mathbb{Z}_{\geq 0}}
\newcommand{\integerspositive}{\mathbb{Z}_{> 0}}

\newcommand{\oprocendsymbol}{\hbox{$\bullet$}}
\newcommand{\oprocend}{\relax\ifmmode\else\unskip\hfill\fi\oprocendsymbol}

\newcommand{\defeq}{\vcentcolon=}

\newcommand{\bv}{{\mathscr{B}}}
\newcommand{\Tini}{{T_{\textup{ini}}}}
\newcommand{\uini}{{u_{\textup{ini}}}}
\newcommand{\yini}{{y_{\textup{ini}}}}

\newcommand{\xini}{{x_{\textup{ini}}}}
\newcommand{\Up}{{U_{\mathrm{p}}}}
\newcommand{\Uf}{{U_{\mathrm{f}}}}
\newcommand{\Yp}{{Y_{\mathrm{p}}}}
\newcommand{\Yf}{{Y_{\mathrm{f}}}}
\newcommand{\Uphat}{{\widehat{U}_{\mathrm{p}}}}
\newcommand{\Ufhat}{{\widehat{U}_{\mathrm{f}}}}
\newcommand{\Yphat}{{\widehat{Y}_{\mathrm{p}}}}
\newcommand{\Yfhat}{{\widehat{Y}_{\mathrm{f}}}}
\newcommand{\col}{{\textup{col}}}

\title{Data-Enabled Predictive Control: In the Shallows of the DeePC}

\author{Jeremy Coulson \qquad John Lygeros \qquad Florian D\"{o}rfler
\thanks{All authors are with the Department of Information Technology and Electrical Engineering at ETH Z\"{u}rich, Switzerland~\texttt{\{jcoulson, lygeros, dorfler\}@control.ee.ethz.ch}. This research was supported by ETH Z\"{u}rich funds and by the ERC under the project OCAL, grant number 797845.}}

\begin{document}

\maketitle
\thispagestyle{empty}
\pagestyle{empty}

\begin{abstract}

We consider the problem of optimal trajectory tracking for unknown systems. A novel data-enabled predictive control (DeePC) algorithm is presented that computes optimal and safe control policies using real-time feedback driving the unknown system along a desired trajectory while satisfying system constraints. Using a finite number of data samples from the unknown system, our proposed algorithm uses a behavioural systems theory approach to learn a non-parametric system model used to predict future trajectories. The DeePC algorithm is shown to be equivalent to the classical and widely adopted Model Predictive Control (MPC) algorithm in the case of deterministic linear time-invariant systems. In the case of nonlinear stochastic systems, we propose regularizations to the DeePC algorithm. Simulations are provided to illustrate performance and compare the algorithm with other methods.

\end{abstract}

\section{Introduction}
As systems are becoming more complex and data is becoming more readily available, scientists and practitioners are beginning to bypass classical model-based techniques in favour of data-driven methods~\cite{future-of-the-field:17,ZSH-ZW:13}. Data-driven methods are suitable for applications where first-principle models are not conceivable (e.g., in human-in-the-loop applications), when models are too complex for control design (e.g., in fluid dynamics), and when thorough modelling and parameter identification is too costly (e.g., in robotics). In fact, it is sometimes easier to learn control policies directly from data, rather than learning a model (the quintessential example supporting this claim is PID control~\cite{KJA-TH:95-book}).

A challenging problem in systems control is optimal trajectory tracking, where a control policy is computed based on output feedback that drives a dynamical system along a desired output trajectory while minimizing a stage cost and respecting safety constraints. A special case of the trajectory tracking problem is regulation, in which a control policy drives the system to an equilibrium point. One of the most celebrated and widely used control techniques for trajectory tracking is receding horizon Model Predictive Control (MPC), precisely because it allows one to include safety considerations during control design~\cite{JBR-DQM:09-book, EFC-CBA:13, AB-MM:99, DQM:14, FB-AB-MM:17-book}. Applications in which trajectory tracking is approached via MPC include autonomous driving~\cite{MB-JF-SE-JCG:17}, autonomous flight~\cite{IP-SO-RB-JBS-CS-SIN:13}, mobile robots~\cite{FB-EF-MP-KS-SLS:11}, and smart energy systems~\cite{GD-AG-RSS-JL:17}, among others. The key ingredient for MPC is an accurate parametric state space model of the system (represented by state space matrices), but obtaining such a model is often the most time-consuming and expensive part of control design~\cite{HH:05, BAO:96,JR:93}. 

In the context of unknown black-box systems, there is no approach which solves the optimal trajectory tracking problem subject to constraints and partial (output) observations. However, some more benign variations of the optimal trajectory tracking problem have been approached using data-driven and learning based methods. We do not attempt to provide a comprehensive survey of the recent, albeit already vast, literature. Rather, we single out a few references that are representative for (and at the forefront of) different approaches in the literature.

We begin with reinforcement learning approaches in which control actions are chosen to maximize a reward~\cite{FL-DV-KGV:12,YO-MG-RJ:17, BK-FLL-HM-AK-MN:14,AMD-SM:17}. This requires either exploring the system via random control actions, or exploiting the knowledge gained by applying optimal control actions. Reinforcement learning approaches usually require a large number of data samples to perform well, and are often very sensitive to hyper-parameters leading to non-reproducible and highly variable outcomes~\cite{BR:18,RI-HP-MG-DP:17}. Additionally, these approaches do not address all of the challenges present in the optimal trajectory tracking problem; namely, they generally do not take into account safety constraints, and typically assume full state feedback.

Other approaches propose performing sequential system identification (ID) and control. System ID can be used to produce an approximate model as well as provide finite sample guarantees quantifying model uncertainty, thus allowing for robust control design~\cite{MCC-EW:02, MV-RLK:08, ST-RB-AP-BR:17}. In this spirit, an end-to-end ID and control pipeline is given in~\cite{RB-NM-BR:18,SD-HM-NM-BR-ST:18} and arrives at a data-driven control solution with guarantees on the sample efficiency, stability, performance, and robustness. The system identification step in these approaches disregards one of the main advantages of a data-driven approach: independence from an underlying parametric representation (e.g., state space representation). In fact, non-parametric learning approaches often outperform parametric approaches (see, e.g., Gaussian processes for regression~\cite{CER:04}). Additionally, they only consider regulation, rely on having full state information, and do not enforce constraint satisfaction. 

Beyond reinforcement learning and sequential ID and control, there are many other adaptive control and safe learning approaches~\cite{FB-MT-AS-AK:17, FJF-AAK-MNZ-SK-JG-CJT:17}. However, they rely on a-priori stabilizing controllers and safe regions, and thus apply only to a small class of problems. 

MPC based on Dynamic Matrix Control has been historically used as a data-driven control technique, in which zero-initial condition step responses are used to predict future trajectories~\cite{RWF-AB:85, CEG-DMP-MM:89}. Although this technique has many limitations~\cite{PL-JHL-MM-SS:95}, it motivates the use of a non-parametric predictive control model. Other non-parametric predictive models have been proposed in~\cite{IM-PR:08, WF-BDM-PVO-MG:99, AA-MJA-DS-DS:18}. These methods do not solve the problem of optimal trajectory tracking with constraints, but serve as building blocks for our approach.

Here we present a \textbf{D}ata-\textbf{e}nabl\textbf{e}d \textbf{P}redictive \textbf{C}ontrol (DeePC) algorithm. Unlike classical MPC and the learning-based control techniques outlined above, the DeePC algorithm does not rely on a parametric system representation. Instead, similar to~\cite{IM-PR:08}, we approach the problem from a behavioural systems theory perspective~\cite{JCW:86, JCW-PR-IM-BDM:05, IM-JCW-SVH-BDM:06-book}. Rather than attempting to learn a parametric system model, we aim at learning the system's ``behaviour'' (see Section~\ref{sec:preliminaries} for the precise definition). Our novel predictive control strategy computes optimal controls for unknown systems using real-time output feedback via a receding horizon implementation, allowing one to incorporate input/output constraints to ensure safety. We formally show the equivalence of the DeePC algorithm to the classical MPC algorithm in the special case of deterministic linear time-invariant (LTI) systems. Our approach is much simpler to implement than the learning-based techniques above, as well as model-based approaches which require system identification and state observer design. In contrast to most MPC formulations (where full state measurement is required), the DeePC algorithm only requires output measurements. Additionally, since our approach does not rely on a parametric system model, we are hopeful that it can also be applied beyond deterministic LTI systems. To apply the algorithm to such systems, we propose some preliminary though insightful regularizations (e.g., low-rank approximations~\cite{SC:15}, and introduction of auxiliary slack variables~\cite{RF:13-book}), and reason as to why these regularizations improve performance. All of the results are validated with a nonlinear and stochastic aerial robotics case study, in which the DeePC algorithm is shown to outperform sequential ID with MPC. 

The remainder of the paper is organized as follows: in Section~\ref{sec:problem}, we formally define the problem. In Section~\ref{sec:MPC} we review the basic notion of MPC. In Section~\ref{sec:preliminaries} we introduce behavioural system theory. Section~\ref{sec:DeePC} contains the DeePC algorithm. In Section~\ref{sec:beyondLTI} we simulate the DeePC algorithm and compare it with sequential ID and MPC using a quadcopter simulation. We conclude the paper in Section~\ref{sec:conclusion}.
\section{Problem Statement}\label{sec:problem}
Consider the discrete-time system given by
\begin{equation}\label{eq:LTIsystem}
\begin{cases}
x(t+1)=Ax(t)+Bu(t) \\
y(t)=Cx(t)+Du(t),
\end{cases}
\end{equation}
where $A\in\real^{n\times n}$, $B\in\real^{n\times m}$, $C\in\real^{p\times n}$, $D\in\real^{p\times m}$, and $x(t)\in \real^n$, $u(t)\in \real^m$, $y(t)\in \real^p$ are respectively the state, control input, and output of the system at time $t\in \integersnonnegative$. Given a desired reference trajectory $r=(r_0,r_1,\dots)\in (\real^p)^{\integersnonnegative}$, input constraint set $\mathcal{U}\subseteq \real^{m}$, output constraint set $\mathcal{Y}\subseteq\real^{p}$, we wish to apply control inputs such that the system output tracks the reference trajectory $r$ while satisfying constraints and optimizing a cost function. Tracking of the trivial trajectory $r=(0,0,\dots)$ is simply regulation.

In the case when the model for the system is \emph{known}, i.e., matrices $A$, $B$, $C$ and $D$ are known, the problem can be approached using MPC (see Section~\ref{sec:MPC}). This paper focuses on the above trajectory tracking problem in the case when the model for system~\eqref{eq:LTIsystem} is \emph{unknown}, but input/output data samples are available.
\section{MPC: A Brief Overview}\label{sec:MPC}

We outline the well-known receding-horizon model predictive control and estimation algorithm for trajectory tracking when the model of system~\eqref{eq:LTIsystem} is known (see, e.g.,~\cite{AF-DL-IA-TA-EFC:09}).
Consider the following optimization problem:
\begin{align}\label{eq:MPC}
\underset{u,x,y}{\textup{minimize}}\quad
& \sum_{k=0}^{N-1}\left(\left\|y_k-r_{t+k}\right\|_Q^2 +\left\|u_{k}\right\|_R^2\right) \nonumber \\
\text{subject to\quad}
& x_{k+1}=Ax_k+Bu_k, \;\forall k \in \{0,\dots, N-1\},\nonumber \\
& y_k=Cx_k+Du_k, \; \forall k \in\{0,\dots, N-1\}, \nonumber \\
& x_0=\hat{x}(t), \\
& u_k\in \mathcal{U},\; \forall k \in \{0,\dots,N-1\}, \nonumber \\
& y_k\in\mathcal{Y}, \; \forall k \in \{0,\dots,N-1\},\nonumber
\end{align}
where $N\in \integerspositive$ is the time horizon, $u=(u_0,\dots,u_{N-1})$, $x=(x_0,\dots,x_{N})$, $y=(y_0,\dots,y_{N-1})$ are the decision variables, and $r_{t+k} \in\real^p$ is the desired reference at time $t+k$, where $t\in \integersnonnegative$ is the time at which the optimization problem should be solved. The norm $\|u_k\|_R$ denotes the quadratic form $u_k^TRu_k$ (similarly for $\|\cdot\|_Q$), where $R\in \real^{m\times m}$ is the control cost matrix and $Q\in\real^{p\times p}$ is the output cost matrix. The estimated state at time $t$ is denoted by $\hat{x}(t)$ and the predicted state and output at time $t+k$ are denoted by $x_k$ and $y_k$, respectively. If the entire state is measured then $\hat{x}(t)=x(t)$. When the state measurement is not available, but system~\eqref{eq:LTIsystem} is observable, an observer is
typically used to estimate the state based on knowledge of the system~\eqref{eq:LTIsystem} and the measured output $y$~\cite{DQM:14}. One may also combine the control and estimation into a single min-max optimization problem~\cite{DAC-JPH:17}. 

The classical MPC algorithm involves solving optimization problem~\eqref{eq:MPC} in a receding horizon manner.
\begin{algorithm}
\caption{MPC}\label{alg:MPC}
\textbf{Input:} $(A,B,C,D)$, reference trajectory $r$, past input/output data $(u,y)$, constraint sets $\mathcal{U}$ and $\mathcal{Y}$, and performance matrices $Q$ and $R$
\begin{enumerate}
\item Generate state estimate $\hat{x}(t)$ using past input/output data.\label{step:MPC1}
\item Solve~\eqref{eq:MPC} for $u^{\star}=(u^{\star}_0,\dots,u^{\star}_{N-1})$. \label{step:MPC2}
\item Apply inputs $(u(t),\dots,u(t+s))=(u^{\star}_0,\dots,u_s^{\star})$ for some $s\leq N-1$.\label{step:MPC3}
\item Set $t$ to $t+s$ and update past input/output data.\label{step:MPC4}
\item Return to~\ref{step:MPC1}.
\end{enumerate}
\end{algorithm}

Note that choosing $s>0$ reduces the number of computations, and in some cases may improve performance~\cite{LG-JP-MS-KW:10}.

Under standard assumptions, it can be shown that the MPC algorithm is recursively feasible and stabilizing~\cite{EFC-CBA:13}. One crucial ingredient for MPC is an accurate model of the system; this is needed both to formulate problem~\eqref{eq:MPC} and, in cases where the state is not measured exactly, to generate the initial state estimates $\hat{x}(t)$. The need for a model is traditionally addressed through system identification~\cite{LL:86-book}, where observations of the system are collected offline before the online control operation begins and are used to estimate a model of the form~\eqref{eq:LTIsystem} that matches the observed data in an appropriate sense.
 For complex systems, this can be a cumbersome and expensive process~\cite{HH:05, BAO:96, JR:93}. 
For this reason, we propose a Data-enabled Predictive Control (DeePC) algorithm which learns the behaviour of the system and does not require an explicit model, system identification, or state estimation (see Algorithm~\ref{alg:DeePC}).

\section{Preliminaries}\label{sec:preliminaries}
\subsection{Non-parametric system representation}

Behavioural system theory is a natural way of viewing a dynamical system when one is not concerned with a particular system representation, but rather the subspace of the signal space in which trajectories of the system live. This is in contrast with classical system theory, where a particular parametric system representation (such as the state-space model~\eqref{eq:LTIsystem}) is used to describe the input/output behaviour, and properties of the system are derived by studying the chosen system representation.
Following~\cite{IM-JCW-SVH-BDM:06-book}, we define a dynamical system and its properties in terms of its behaviour.
\begin{definition}
A \emph{dynamical system} is a $3$-tuple $(\integersnonnegative,\mathbb{W},\bv)$ where $\integersnonnegative$ is the discrete-time axis, $\mathbb{W}$ is a signal space, and $\bv\subseteq \mathbb{W}^{\integersnonnegative}$ is the behaviour.
\end{definition}
\begin{definition}
Let $(\integersnonnegative,\mathbb{W},\bv)$ be a dynamical system.
\begin{enumerate}[(i)]
\item $(\integersnonnegative,\mathbb{W},\bv)$ is \emph{linear} if $\mathbb{W}$ is a vector space and $\bv$ is a linear subspace of $\mathbb{W}^{\integersnonnegative}$.\label{item:systemdef1}
\item $(\integersnonnegative,\mathbb{W},\bv)$ is \emph{time invariant} if $\bv \subseteq \sigma\bv$ where $\sigma \colon \mathbb{W}^{\integersnonnegative}\to \mathbb{W}^{\integersnonnegative}$ is the forward time shift defined by $(\sigma w)(t)=w(t+1)$ and $\sigma\bv=\{\sigma w \mid w\in \bv\}$.\label{item:systemdef2}
\item $(\integersnonnegative,\mathbb{W},\bv)$ is \emph{complete} if $\bv$ is closed in the topology of pointwise convergence. \label{item:systemdef3}
\end{enumerate}
\end{definition}
\smallskip
Note that if a dynamical system satisfies~\ref{item:systemdef1}-\ref{item:systemdef2} then~\ref{item:systemdef3} is equivalent to finite dimensionality of $\mathbb{W}$ (see~\cite[Section 7.1]{IM-JCW-SVH-BDM:06-book}). We denote the class of systems $(\integersnonnegative,\real^{m+p},\bv)$ satisfying~\ref{item:systemdef1}-\ref{item:systemdef3} by $\mathcal{L}^{m+p}$, where $m,p \in \integersnonnegative$. With slight abuse of notation and terminology, we denote a dynamical system in $\mathcal{L}^{m+p}$ only by its behaviour $\bv$.

Next, we define the set $\bv_T=\{w\in(\real^{m+p})^T \mid \exists\; v\in \bv\; \text{s.t.}\; w_t=v_t,\; 1\leq t\leq T\}$ of trajectories truncated to a window of length $T$. Without loss of generality, we assume that $\bv$ can be written as the product space of two sub-behaviours $\bv^u$ and $\bv^y$, where $\bv^u=(\real^m)^{\integersnonnegative}$ and $\bv^y \subseteq (\real^p)^{\integersnonnegative}$ are the spaces of input and output signals, respectively (see~\cite[Theorem 2]{JCW:86}), that is, any trajectory $w\in \bv$ can be written as $w=\col(u,y)$, where $\col(u,y)\defeq (u^T,y^T)^T$.
We now present two concepts: controllability, and persistency of excitation.
\begin{definition}\label{def:controllable}
A system $\bv\in\mathcal{L}^{m+p}$ is \emph{controllable} if for every $T\in \integerspositive$, $w^1\in \bv_T$, $w^2\in\bv$ there exists $w\in \bv$ and $T'\in \integerspositive$ such that $w_t=w^1_t$ for $1\leq t \leq T$ and $w_t=w^2_{t-T-T'}$ for $t>T+T'$.
\end{definition}
\smallskip
In other words, a behavioural system is controllable if any two trajectories can be patched together in finite time.
\begin{definition}
Let $L,T\in\integerspositive$ such that $T\geq L$. The signal $u=\col(u_1,\dots,u_T)\in \real^{Tm}$ is \emph{persistently exciting of order $L$} if the Hankel matrix 
\[
\mathscr{H}_L(u)\defeq
\begin{pmatrix}
&u_1 &\cdots &u_{T-L+1} \\
&\vdots &\ddots &\vdots \\
&u_L &\cdots &u_T
\end{pmatrix}
\]
is of full row rank.
\end{definition}
\smallskip
The term persistently exciting describes an input signal sufficiently rich and long as to \emph{excite} the system yielding an output sequence that is representative for the system's behaviour.

\subsection{Parametric system representation}
There are several equivalent ways of representing a behavioural system $\bv \in \mathcal{L}^{m+p}$, including the classical input/output/state representation~\eqref{eq:LTIsystem} denoted by $\bv(A,B,C,D)=\{\col(u,y) \in (\real^{m+p})^{\integersnonnegative}\mid \exists\; x\in(\real^n)^{\integersnonnegative}\; \text{s.t.}\;  \sigma x=Ax+Bu,\; y=Cx+Du\}$. The input/output/state representation of smallest order (i.e., smallest state dimension) is called a \emph{minimal representation}, and we denote its order by $\bm{n}(\bv)$. Another important property of a system $\bv\in\mathcal{L}^{m+p}$ is the \emph{lag} defined by the smallest integer $\ell\in \integerspositive$ such that the observability matrix $\mathscr{O}_{\ell}(A,C)\defeq \col\left(C,CA, \dots,CA^{\ell-1}\right)$ has rank $\bm{n}(\bv)$. We denote the lag by $\bm{\ell}(\bv)$ (see~\cite[Section 7.2]{IM-JCW-SVH-BDM:06-book} for equivalent input/output/state representation free definitions of lag). The lower triangular Toeplitz matrix consisting of $A,B,C,D$ is denoted by 
\[
\mathscr{T}_N(A,B,C,D) \defeq \begin{pmatrix}
D &0 &\cdots &0 \\
CB &D &\cdots &0 \\
\vdots &\ddots &\ddots &\vdots \\
CA^{N-2}B &\cdots &CB &D
\end{pmatrix}.
\]
We can now present a uniqueness result.
\begin{lemma}\longthmtitle{\hspace{1sp}\cite[Lemma 1]{IM-PR:08}}\label{lem:initialstate}
Let $\bv \in \mathcal{L}^{m+p}$ and $\bv(A,B,C,D)$ a minimal input/output/state representation. Let $\Tini,N \in \integerspositive$ with $\Tini \geq \bm{\ell}(\bv)$ and $\col(\uini,u,\yini,y)\in \bv_{\Tini+N}$. Then there exists a unique $\xini \in \real^{\bm{n}(\bv)}$ such that
\begin{equation}\label{eq:initialstate}
y=\mathscr{O}_N(A,C)\xini +\mathscr{T}_N(A,B,C,D)u.
\end{equation}
\end{lemma}
\smallskip
In other words, given a sufficiently long window of initial system data $\col(\uini,\yini)$, the state to which the system is driven by the sequence of inputs $\uini$ is unique. Furthermore, if the matrices $A,B,C,D$ are known, the state $\xini$ can be computed. We now present a result known in behavioural systems theory as the \emph{Fundamental Lemma}~\cite{IM-PR:08}.%
\begin{lemma}\longthmtitle{\hspace{1sp}\cite[Theorem 1]{JCW-PR-IM-BDM:05}}\label{lem:fundamental}
Consider a controllable system $\bv\in\mathcal{L}^{m+p}$. Let $T,t\in\integerspositive$, and $w=\col(u,y)\in \bv_T$. Assume $u$ to be persistently exciting of order $t+\bm{n}(\bv)$. Then $\textup{colspan}(\mathscr{H}_t(w))=\bv_t$.
\end{lemma}
\smallskip
Note that the number of data points $T$ must be at least $(m+1)(t+\bm{n}(\bv))-1$ in order to satisfy the persistency of excitation condition. Lemma~\ref{lem:fundamental} replaces the need for a model or system identification process and allows for any trajectory of a controllable LTI system to be constructed using a finite number of data samples generated by a  sufficiently rich (in particular, persistently exciting) input signal. In a sense, the Hankel matrix is itself a non-parametric predictive model based on raw data. It allows one to implicitly estimate the state of an LTI system, predict its future behaviour, and design optimal feedforward control inputs~\cite{IM-PR:08}. Lemma~\ref{lem:initialstate} and Lemma~\ref{lem:fundamental} can be used together in a predictive control algorithm similar to Algorithm~\ref{alg:MPC} by replacing the state observer (respectively, the system model) with the data $\col(\uini,\yini)$ (respectively, $\mathscr{H}_t(w)$).

\section{DeePC: A data-enabled predictive control algorithm}\label{sec:DeePC}
\subsection{Data collection}
We begin by assuming that the data is generated by an unknown controllable LTI system $\bv\in\mathcal{L}^{m+p}$ with a minimal input/output/state representation $\bv(A,B,C,D)$. Let $T,\Tini,N \in \integerspositive$ such that $T\geq (m+1)(\Tini+N+\bm{n}(\bv))-1$. Let $u^{\textup{d}}=\col(u_1^{\textup{d}},\dots,u_T^{\textup{d}})\in \real^{Tm}$ be a sequence of $T$ inputs applied to $\bv$, and $y^{\textup{d}}=(y_1^{\textup{d}},\dots,y_T^{\textup{d}})\in \real^{Tp}$ the corresponding outputs. Furthermore, assume $u^{\textup{d}}$ is persistently exciting of order $\Tini+N+\bm{n}(\bv)$. The superscript $\textup{d}$ is used to indicate that these are sequences of data samples collected during an offline procedure from the unknown system. Note that the data $\col(u^{\textup{d}},y^{\textup{d}})\in\bv_T$ can be equivalently thought of as coming from the minimimal input/output/state representation $\bv(A,B,C,D)$. Next, we partition the input/output data into two parts which we call \emph{past data} and \emph{future data}. More formally, define
\begin{equation}\label{eq:UpUfYpYf}
\begin{pmatrix}
\Up \\ \Uf 
\end{pmatrix}\defeq \mathscr{H}_{\Tini+N}(u^{\textup{d}}), \quad
\begin{pmatrix}
\Yp \\ \Yf 
\end{pmatrix}\defeq \mathscr{H}_{\Tini+N}(y^{\textup{d}}),
\end{equation}
where $\Up$ consists of the first $\Tini$ block rows of $\mathscr{H}_{\Tini+N}(u^{\textup{d}})$ and $\Uf$ consists of the last $N$ block rows of $\mathscr{H}_{\Tini+N}(u^{\textup{d}})$ (similarly for $\Yp$ and $\Yf$).
In the sequel, past data denoted by the subscript $\mathrm{p}$ will be used to estimate the initial condition of the underlying state, whereas the future data denoted by the subscript $\mathrm{f}$ will be used to predict the future trajectories.

\subsection{State estimation and trajectory prediction}
By Lemma~\ref{lem:fundamental}, we can construct any $\Tini+N$ length trajectory of $\bv_{\Tini+N}$ using the data collected. Indeed, a trajectory $\col(\uini,u,\yini,y)$ belongs to $\bv_{\Tini+N}$ if and only if there exists $g\in \real^{T-\Tini-N+1}$ such that
\begin{equation}\label{eq:datamodel}
\begin{pmatrix}
\Up \\ \Yp \\ \Uf \\ \Yf 
\end{pmatrix}
g=
\begin{pmatrix}
\uini \\ \yini \\ u \\ y
\end{pmatrix}.
\end{equation}
If $\Tini \geq \bm{\ell}(\bv)$, then Lemma~\ref{lem:initialstate} implies that there exists a unique $\xini \in \real^{\bm{n}(\bv)}$ such that the output $y$ is uniquely determined by~\eqref{eq:initialstate}.
Intuitively, the trajectory $\col(\uini,\yini)$ fixes the underlying initial state $\xini$ from which the trajectory $\col(u,y)$ evolves. Note, however, that~\eqref{eq:datamodel} does not require the input/output/state representation of the system to be known. The state $\xini$ is only ``fixed'' implicitly by $\col(\uini, \yini)$. As first shown in~\cite{IM-PR:08}, this allows one to predict future trajectories based on a given initial trajectory $\col(\uini,\yini)\in\bv_{\Tini}$, and the precollected data in $\Up$, $\Uf$, $\Yp$, and $\Yf$. Indeed, given an initial trajectory $\col(\uini,\yini)\in\bv_{\Tini}$ of length $\Tini \geq \bm{\ell}(\bv)$ and a sequence of future inputs $u\in \real^{Nm}$, the first three block equations of~\eqref{eq:datamodel} can be solved for $g$. The sequence of future outputs are then given by $y=\Yf g$. Furthermore, by Lemma~\ref{lem:initialstate} the vector $y$ computed contains the unique sequence of outputs corresponding to the inputs $u$. Vice versa, given a desired reference output $y$ an associated feedforward control input can be calculated~\cite{IM-PR:08}.

\subsection{DeePC algorithm}
Given a time horizon $N \in \integerspositive$, a reference trajectory $r=(r_0,r_1,\dots)\in (\real^{p})^{\integersnonnegative}$, past input/output data $\col(\uini,\yini)\in \bv_{\Tini}$, input constraint set $\mathcal{U}\subseteq \real^m$, output constraint set $\mathcal{Y}\subseteq \real^p$, output cost matrix $Q\in \real^{p\times p}$, and control cost matrix $R \in \real^{m\times m}$, we formulate the following optimization problem:
\begin{align}\label{eq:DeePC}
\underset{g,u,y}{\text{minimize}}\quad
& \sum_{k=0}^{N-1}\left(\left\|y_k-r_{t+k}\right\|_Q^2 +\left\|u_{k}\right\|_R^2\right)\nonumber \\
\text{subject to\quad}
& \begin{pmatrix}
\Up \\ \Yp \\ \Uf \\ \Yf
\end{pmatrix}g
=\begin{pmatrix}
\uini \\ \yini \\ u \\ y
\end{pmatrix}, \\
& u_k\in \mathcal{U}, \; \forall k \in \{0,\dots,N-1\}, \nonumber \\
& y_k\in\mathcal{Y},\; \forall k \in \{0,\dots,N-1\}.\nonumber
\end{align}
Note here, that $u$ and $y$ are not independent decision variables of the optimization problem. Rather they are described completely by the fixed data matrices $\Uf$ and $\Yf$ and the decision variable $g$. A comparison of the two optimal control problems~\eqref{eq:MPC} and~\eqref{eq:DeePC} yields only a single (though key) difference; the model and the state estimate in~\eqref{eq:MPC} are replaced completely with input/output data samples in~\eqref{eq:DeePC}. We now present the DeePC algorithm.
\begin{algorithm}
\caption{DeePC}\label{alg:DeePC}
\textbf{Input:} $\col(u^{\textup{d}},y^{\textup{d}})\in \bv_T$, reference trajectory $r\in\real^{Np}$, past input/output data $\col(\uini,\yini)\in\bv_{\Tini}$, constraint sets $\mathcal{U}$ and $\mathcal{Y}$, and performance matrices $Q$ and $R$
\begin{enumerate}
\item Solve~\eqref{eq:DeePC} for $g^{\star}$.\label{step:DeePCbegin}
\item Compute the optimal input sequence $u^{\star}=\Uf g^{\star}$.
\item Apply input $(u(t),\dots,u(t+s))=(u_0^{\star},\dots,u_s^{\star})$ for some $s\leq N-1$.
\item Set $t$ to $t+s$ and update $\uini$ and $\yini$ to the $\Tini$ most recent input/output measurements.\label{step:DeePCend}
\item Return to~\ref{step:DeePCbegin}.
\end{enumerate}
\end{algorithm}
\subsection{Equivalence of DeePC and MPC}
It can be shown that Algorithm~\ref{alg:MPC} and Algorithm~\ref{alg:DeePC} yield equivalent closed-loop trajectories under certain assumptions. We first show the equivalence of the feasible sets of~\eqref{eq:MPC} and~\eqref{eq:DeePC}.
\begin{theorem}\longthmtitle{Feasible Set Equivalence}\label{thm:main}
Consider a controllable LTI system $\bv \in \mathcal{L}^{m+p}$ with minimal input/ouput/state representation $\bv(A,B,C,D)$ given as in~\eqref{eq:LTIsystem}. Consider the MPC and DeePC optimization problems~\eqref{eq:MPC} and~\eqref{eq:DeePC}. Let $\Tini \geq \bm{\ell}(\bv)$ and $\col(\uini ,\yini)\in \bv_{\Tini}$ be the most recent input/output measurements from system~\eqref{eq:LTIsystem}. Assume that the data $\col(u^{\textup{d}},y^{\textup{d}})\in\bv_T$ in $\col(\Up,\Yp,\Uf,\Yf)$ is such that $u^{\textup{d}}$ is persistently exciting of order $\Tini+N+\bm{n}(\bv)$, where $T\geq(m+1)(\Tini+N+\bm{n}(\bv))-1$. Then there exists a state estimate $\hat{x}(t)$ in~\eqref{eq:MPC} such that the feasible sets of~\eqref{eq:MPC} and~\eqref{eq:DeePC} are equal.
\end{theorem}
\begin{proof}
We first look at the feasible set of~\eqref{eq:DeePC}. Since $u^{\textup{d}}$ is persistently exciting of order $\Tini+N+\bm{n}(\bv)$, then by Lemma~\ref{lem:fundamental} $\textup{image}\left(\col(\Up,\Yp,\Uf,\Yf)\right)=\bv_{\Tini+N}$, where $\Up$, $\Yp$, $\Uf$, $\Yf$ are defined as in~\eqref{eq:UpUfYpYf}. Hence, the feasible set of~\eqref{eq:DeePC} is equal to $\{(u,y)\in \mathcal{U}^N\times\mathcal{Y}^N\mid \col(\uini,u,\yini,y)\in\bv_{\Tini+N}\}$, where $\mathcal{U}^N$ is the cartesian product of $\mathcal{U}$ with itself $N$-times (similarly for $\mathcal{Y}^N$). Since the system $\bv$ yields an equivalent representation given by $\bv(A,B,C,D)$, then by Lemma~\ref{lem:initialstate} the feasible set of~\eqref{eq:DeePC} can be written as the set of pairs $(u,y)\in \mathcal{U}^N\times\mathcal{Y}^N$ satisfying
\[
y= \mathscr{O}_N(A,C)\xini+\mathscr{T}_N(A,B,C,D) u,
\]
where $\xini$ is uniquely determined from $\col(\uini,\yini)$. 
We now look at the feasible set of~\eqref{eq:MPC}. By rewriting the constraints in~\eqref{eq:MPC} we obtain
\[
y= \mathscr{O}_N(A,C)\hat{x}(t)+\mathscr{T}_N(A,B,C,D) u,
\quad u\in\mathcal{U}^N,\; y\in\mathcal{Y}^N,
\]
where $\hat{x}(t)$ is the estimation of the state $x(t)$ at time $t$. Setting the state estimate $\hat{x}(t)=\xini$ yields equal feasible sets since the state-space coordinates of the $\bv(A,B,C,D)$ and the system in~\eqref{eq:MPC} are identical.
\end{proof}

Note that the state estimate $\hat{x}(t)=\xini$ is a natural choice when full-state measurements are available or when input/output measurements are deterministic.
\begin{corollary}\longthmtitle{Equivalent Closed Loop Behaviour}\label{cor:main}
Consider the MPC Algorithm~\ref{alg:MPC} and the DeePC Algorithm~\ref{alg:DeePC} with $Q\succeq 0$, $R\succ 0$, and $\mathcal{U}$, $\mathcal{Y}$ convex and non-empty. Under the assumptions of Theorem~\ref{thm:main}, Algorithm~\ref{alg:MPC} and Algorithm~\ref{alg:DeePC} result in equivalent closed-loop behaviour, i.e., the optimal control sequence $u^{\star}$ and corresponding system output $y^{\star}$ at every iteration is identical.
\end{corollary}
\begin{proof}
Since $R \succ 0$, the cost function in~\eqref{eq:MPC} is strictly convex in the decision variable $u$. Thus, since the constraints are convex and non-empty, a solution $(u_{\textup{MPC}}^{\star},x_{\textup{MPC}}^{\star},y_{\textup{MPC}}^{\star})$ to~\eqref{eq:MPC} exists, and $u_{\textup{MPC}}^{\star}$ is unique (see, e.g.,~\cite{SB-LV:04-book}). Similarly, the cost function in~\eqref{eq:DeePC} is strictly convex in the decision variable $u$ and the constraints are convex and non-empty. Hence, a solution $(g_{\textup{DeePC}}^{\star},u_{\textup{DeePC}}^{\star},y_{\textup{DeePC}}^{\star})$ to~\eqref{eq:DeePC} exists, and $u_{\textup{DeePC}}^{\star}$ is unique. Since the cost function in~\eqref{eq:MPC} and~\eqref{eq:DeePC} coincide and the feasible sets of~\eqref{eq:MPC} and~\eqref{eq:DeePC} are equal (by Theorem~\ref{thm:main}), then $u_{\textup{MPC}}^{\star}=u_{\textup{DeePC}}^{\star}$. Applying control inputs $(u(t),\dots,u(t+s))=(u_0^{\star},\dots,u_s^{\star})$ for some $s\leq N-1$ to system~\eqref{eq:LTIsystem} yields corresponding output sequence $(y(t),\dots,y(t+s)=(y_0^{\star},\dots,y_s^{\star})$. Updating $\col(\uini,\yini)$ to the most recent input/output measurements and setting the state estimate in Algorithm~\ref{alg:MPC} to $\xini$ as in Theorem~\ref{thm:main} yields equal feasible sets. Repeating the above argument, both algorithms compute an identical optimal control sequence. This argument can be repeated for all iterations of the algorithms proving the result.
\end{proof}
%

One notable feature of the DeePC algorithm presented in Algorithm~\ref{alg:DeePC} is its simplicity when compared to reinforcement learning approaches~\cite{BR:18}, and other related model-based schemes. The DeePC algorithm achieves system ID, state estimation, and trajectory prediction with one linear equation resulting in a quadratic program with $T-\Tini-N+1$ number of decision variables where $T\in\integerspositive$ is the amount of data collected. In order to satisfy the persistency of excitation assumption in Theorem~\ref{thm:main}, one must collect a minimum of $(m+1)(\Tini+N+\bm{n}(\bv))-1$ data samples implying that the number of decision variables in~\eqref{eq:DeePC} is at least $m(\Tini+N)+(m+1)\bm{n}(\bv)$.

Note that the persistency of excitation assumption assumes knowledge of $\bm{n}(\bv)$ and $\bm{\ell}(\bv)$, which are properties that are a priori unknown. We know that $\bm{\ell}(\bv)\leq\bm{n}(\bv)$ by the Cayley-Hamilton theorem. Hence an upper bound for $\bm{n}(\bv)$ is sufficient. 
In practice, one would simply collect a sufficiently large amount of data to exceed the necessary amount for persistency of excitation. If, however, $\bm{n}(\bv)$ is underestimated and an insufficient amount of data is collected, the matrix $\col(\Up,\Yp,\Uf,\Yf)$ will represent a reduced order linear system with an approximate input/output behaviour. The precise implications of this model reduction need to be investigated in future work.
\section{Beyond deterministic LTI systems}\label{sec:beyondLTI}
In this section we provide preliminary results which shows the promising extension of the DeePC algorithm beyond deterministic LTI systems. We offer insightful algorithmic extensions by means of salient regularizations, show their utility through a numerical study, and provide plausible reasoning for the regularizations.

\subsection{Regularized DeePC Algorithm}
Consider now the nonlinear discrete-time system given by
\begin{equation}\label{eq:generalsystem}
\begin{cases}
x(t+1)=f(x(t),u(t)) \\
y(t)=h(x(t),u(t),\eta(t)),
\end{cases}
\end{equation}
where $\eta(t)\in \real^p$ is white measurement noise, and $f\colon \real^n\times\real^m\to\real^n$ and $h\colon \real^n \times \real^m\times \real^p \to \real^p$ are not necessarily linear. One may also consider a system affected by process noise. However, we focus on systems only affected by measurement noise in order to isolate its effect on the DeePC algorithm. To apply the DeePC algorithm to system~\eqref{eq:generalsystem}, we propose three regularizations to the optimal control problem~\eqref{eq:DeePC}. In particular, we introduce the following \emph{regularized} optimization problem:
\begin{align}\label{eq:regularizedDeePC}
\underset{g,u,y,\sigma_y}{\text{minimize}}\quad
& \displaystyle\sum_{k=0}^{N-1}\left(\left\|y_k-r_{t+k}\right\|_Q^2 +\left\|u_{k}\right\|_R^2\right) \nonumber \\
& +\lambda_g\|g\|_1 +\lambda_y \|\sigma_y\|_1 \nonumber \\
\text{subject to}\quad
& \begin{pmatrix}
\Uphat \\ \Yphat \\ \Ufhat \\ \Yfhat
\end{pmatrix}g
=\begin{pmatrix}
\uini \\ \yini \\ u \\ y
\end{pmatrix} + 
\begin{pmatrix}
0 \\ \sigma_y \\ 0 \\ 0
\end{pmatrix}, \\
& u_k\in \mathcal{U}, \; \forall k \in \{0,\dots,N-1\}, \nonumber\\
& y_k\in\mathcal{Y},\; \forall k \in \{0,\dots,N-1\},\nonumber
\end{align}
where $\sigma_y \in \real^{\Tini p}$ is an auxiliary slack variable, $\lambda_y,\lambda_g \in \real_{>0}$ are regularization parameters, and $\col(\Uphat, \Yphat, \Ufhat, \Yfhat)$ is a {low-rank matrix approximation} of $\col(\Up,\Yp,\Uf,\Yf)$. 
Let us explain these three regularizations. We offer convincing numerical evidence in Section~\ref{sec:casestudy}.

\textbf{Slack variable:} When the output measurements are corrupted by noise, the constraint equation in~\eqref{eq:DeePC} may become inconsistent. Hence, in~\eqref{eq:regularizedDeePC} we include the slack variable $\sigma_y$ in the constraint to ensure feasibility of the constraint at all times. We penalize the slack variable with a weighted one-norm penalty function. Choosing $\lambda_y$ sufficiently large gives the desired property that $\sigma_y \neq 0$ only if the constraint is infeasible (see, e.g.,~\cite{RF:13-book}), that is, only if the data is inconsistent.

\textbf{One-norm regularization on $\boldsymbol{g}$:} The cost includes a one-norm penalty on $g$. We conjecture that this regularization is related to distributionally robust optimization problems, in which similar regularization terms arise~\cite{PME-DK:18,SSA-DK-PME:17}.

\textbf{Low-rank approximation:} By performing a low-rank matrix approximation (e.g., via singular value decomposition (SVD) and truncation~\cite{SC:15}), we take into account only the most dominant sub-behaviour (corresponding to the largest singular values), resulting in a data matrix describing the behaviour of the closest deterministic LTI system (where ``closest'' is measured with the Frobenius norm in the SVD case). In the case of noisy measurements, the SVD filters the noise. In the case of nonlinear dynamics (which can be lifted to infinite-dimensional linear dynamics with a nonlinear output map~\cite{SB-LC:85,KK-WHS:91-book}), the SVD results in a matrix describing an approximate LTI model, i.e., the most relevant basis functions of the infinite-dimensional lift whose dimension can be chosen by adjusting the SVD cutoff. Note that after performing the low-rank approximation, the DeePC algorithm does not require that the matrix $\col(\Uphat, \Yphat, \Ufhat, \Yfhat)$ have a Hankel structure. This is in contrast to subspace ID techniques, in which low-rank approximations must be carefully modified in order to preserve the Hankel structure of the data matrix resulting in higher algorithmic and computational complexity~\cite{IM:18-book}.

\subsection{Aerial Robotics Case Study}\label{sec:casestudy}
We illustrate the performance of the regularized DeePC algorithm, i.e., Algorithm~\ref{alg:DeePC} with~\eqref{eq:regularizedDeePC}, by simulating it on a high-fidelity nonlinear quadcopter model~\cite{RM-VK-PC:12}, and compare the performance to system identification (ID) followed by MPC using the identified model. The states of the quadcopter model are given by the 3 spatial coordinates ($x$, $y$, $z$) and their velocities, and  the 3 angular coordinates $(\alpha,\beta,\gamma)$ and their velocities, i.e., the state is $(x,y,z,\dot{x},\dot{y},\dot{z},\alpha,\beta,\gamma,\dot{\alpha},\dot{\beta},\dot{\gamma})$. The inputs are given by the thrusts from the 4 rotors, $(u_1,u_2,u_3,u_4)$. We assume full state measurement to facilitate the comparison to standard MPC. Data was collected from the nonlinear model subject to additive white measurement noise. We collected 214 input/output measurements with a sample time of 0.1 seconds. Drawing the input sequence from a uniformly distributed random variable ensured that the data was persistently exciting. For the model-based MPC, the data was used to identify the system parameters through the least square prediction error method with an assumed state dimension of 12. The following parameters were chosen for the optimization problems~\eqref{eq:MPC} and~\eqref{eq:regularizedDeePC}: $N=30$, $\Tini =1$, $R=I$, $Q=\textup{diag}(200,200,300,1,\dots,1)$, $\lambda_g=30$, $\lambda_y=10^5$, $\col(\Uphat,\Yphat,\Ufhat,\Yfhat)=\col(\Up,\Yp,\Uf,\Yf)$. The thrust from each rotor was constrained between $0$ and $1$, and the $(x,y,z)$ coordinates were constrained between $-3$ and $3$.
 
We simulated the system ID followed by the MPC algorithm and the regularized DeePC algorithm on the nonlinear and stochastic quadcopter model in which the quadcopter was commanded to follow a series of figure-eight trajectories for a duration of 60 seconds. We observe that DeePC performs better than sequential ID and MPC in terms of reference tracking and constraint satisfaction; see Figure~\ref{fig:quadplot} for an illustration.

Another simulation was performed in which the quadcopter was commanded to perform a simple step trajectory in the $(x,y,z)$ coordinates with the same constraints listed above. The duration of constraint violations and the cost were measured. This was repeated 30 times with different data sets for constructing $\col(\Up,\Yp,\Uf,\Yf)$, and different random seeds for the measurement noise. The results are displayed in Figure~\ref{fig:hist} and show that DeePC consistently outperforms identification-based MPC in terms of cost and constraint satisfaction. While these observations should be made cautiously, an intuitive explanation for the superior performance of DeePC is that~\eqref{eq:regularizedDeePC} simultaneously optimizes for the best system model, state estimation, and control policy, whereas conventional MPC requires fixing a system model and performs these tasks independently.

To study the effect of the regularizations on the performance of the DeePC algorithm, we performed a sensitivity analysis on the regularization parameters $\lambda_y$ and $\lambda_g$. We did not perform any low-rank approximation to the data. The quadcopter was commanded to follow the same step trajectory as in the previous simulation. The duration of constraint violations and cost were measured. This was repeated 8 times with different data sets, and the duration of constraint violations and cost were averaged over these 8 data sets. The results in Figure~\ref{fig:regularizations} show that the regularizations improve performance.

Our preliminary simulations suggest that one-norm regularization of $\lambda_g$ is more effective and robust than a low-rank approximation of the Hankel matrix $\col(\Up,\Yp,\Uf,\Yf)$. In fact, the latter appears to be sensitive and needs to be performed on a case by case basis to avoid unstable behaviour. This will be investigated in future work.
\begin{figure}[htb!]
 \begin{tabular}{cc}
 \multirow{2}{*}[10mm]{\begin{subfigure}{0.4\columnwidth}
    \centering  \includegraphics[width=\textwidth, trim={0.1pt 0 0 0}, clip=true]{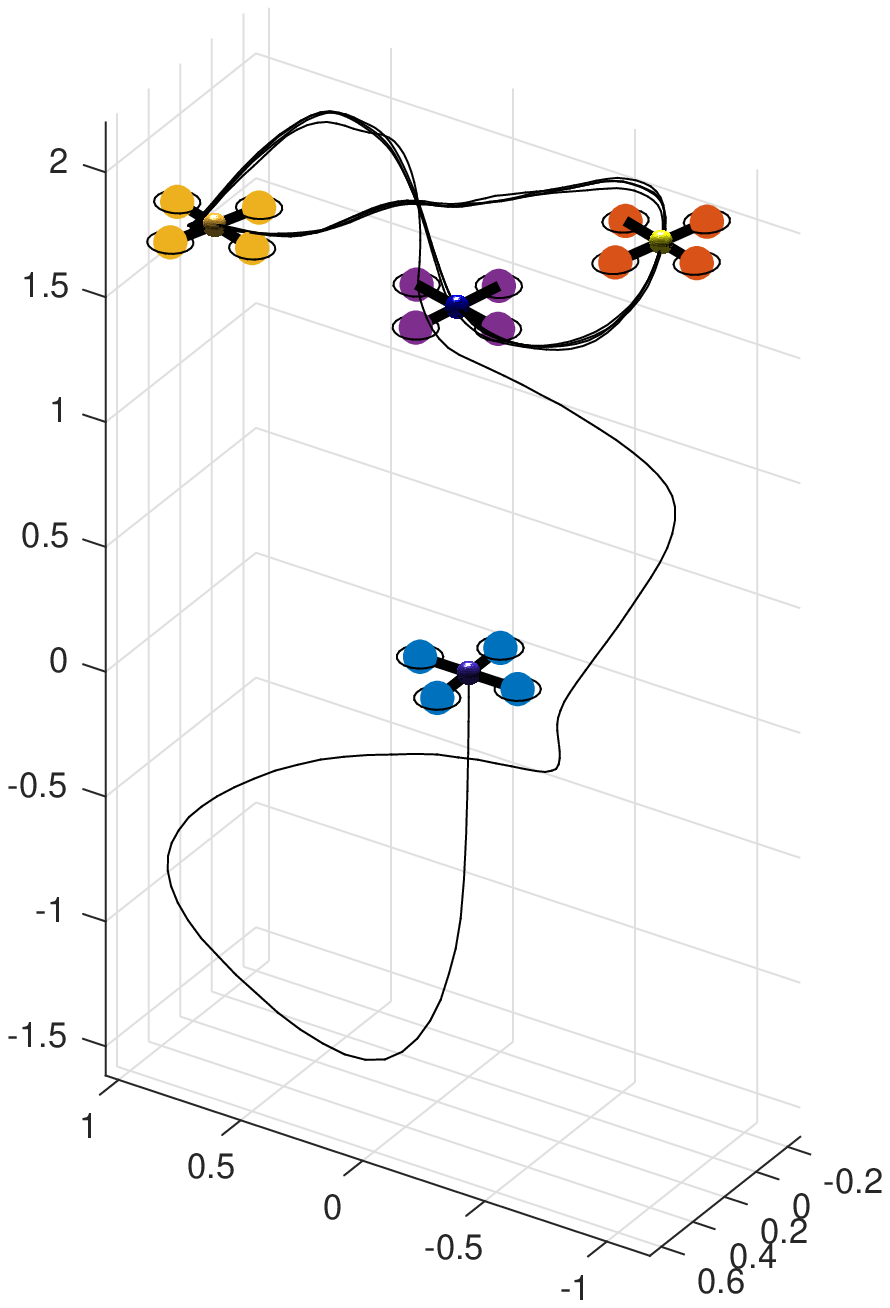}
      \caption{}
      \label{fig:figure8}
    \end{subfigure}}
 &\begin{subfigure}{0.56\columnwidth}
        \includegraphics[width=\textwidth]{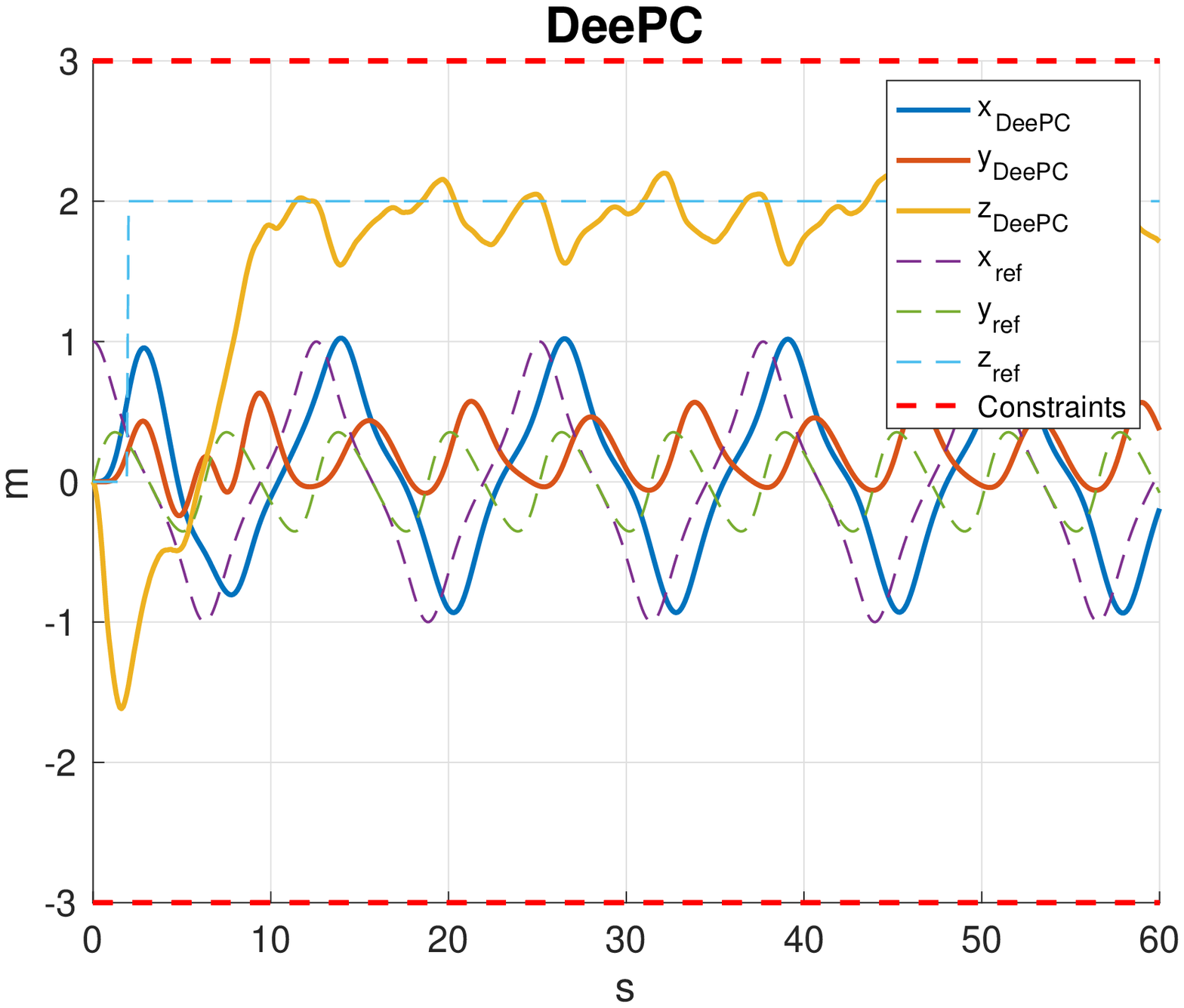}
        \caption{}
        \label{fig:deepc_xyz}
      \end{subfigure}  
\\
&\begin{subfigure}{0.56\columnwidth}
        \includegraphics[width=\textwidth]{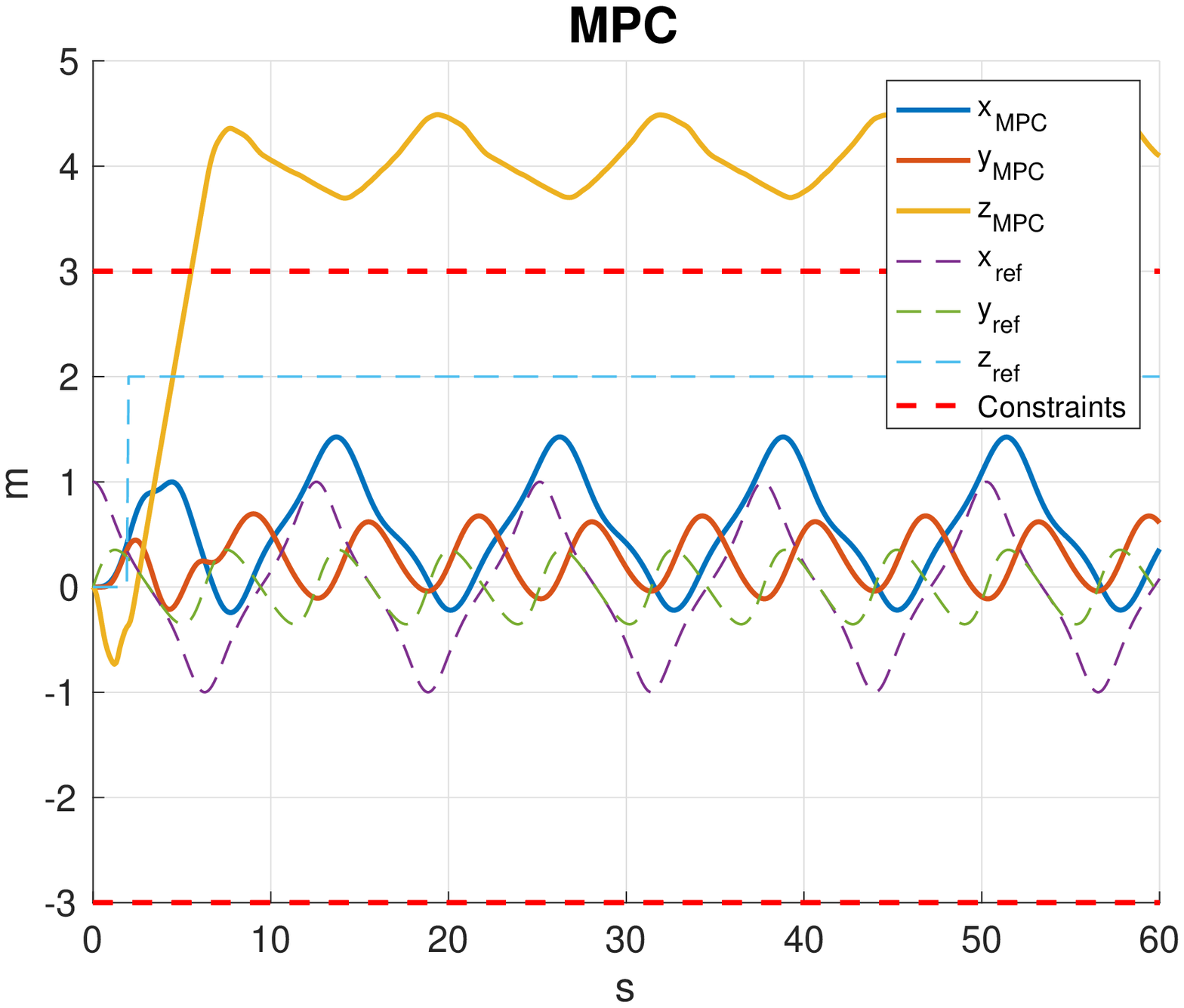}
        \caption{}
        \label{fig:mpc_xyz}
      \end{subfigure}
      
 \end{tabular}
\caption{Figure~\subref*{fig:figure8}: three dimensional plot of the trajectory of the quadcopter at different instances of time controlled with DeePC. Figure~\subref*{fig:deepc_xyz} and~\subref*{fig:mpc_xyz}: trajectories of the spatial coordinates when controlled by DeePC and MPC, respectively. The horizontal red dashed lines represent constraints.}
\label{fig:quadplot}
\end{figure}
\begin{figure}[h!]
	\centering
		\begin{minipage}[h]{0.9\linewidth}
			\includegraphics[width=\linewidth]{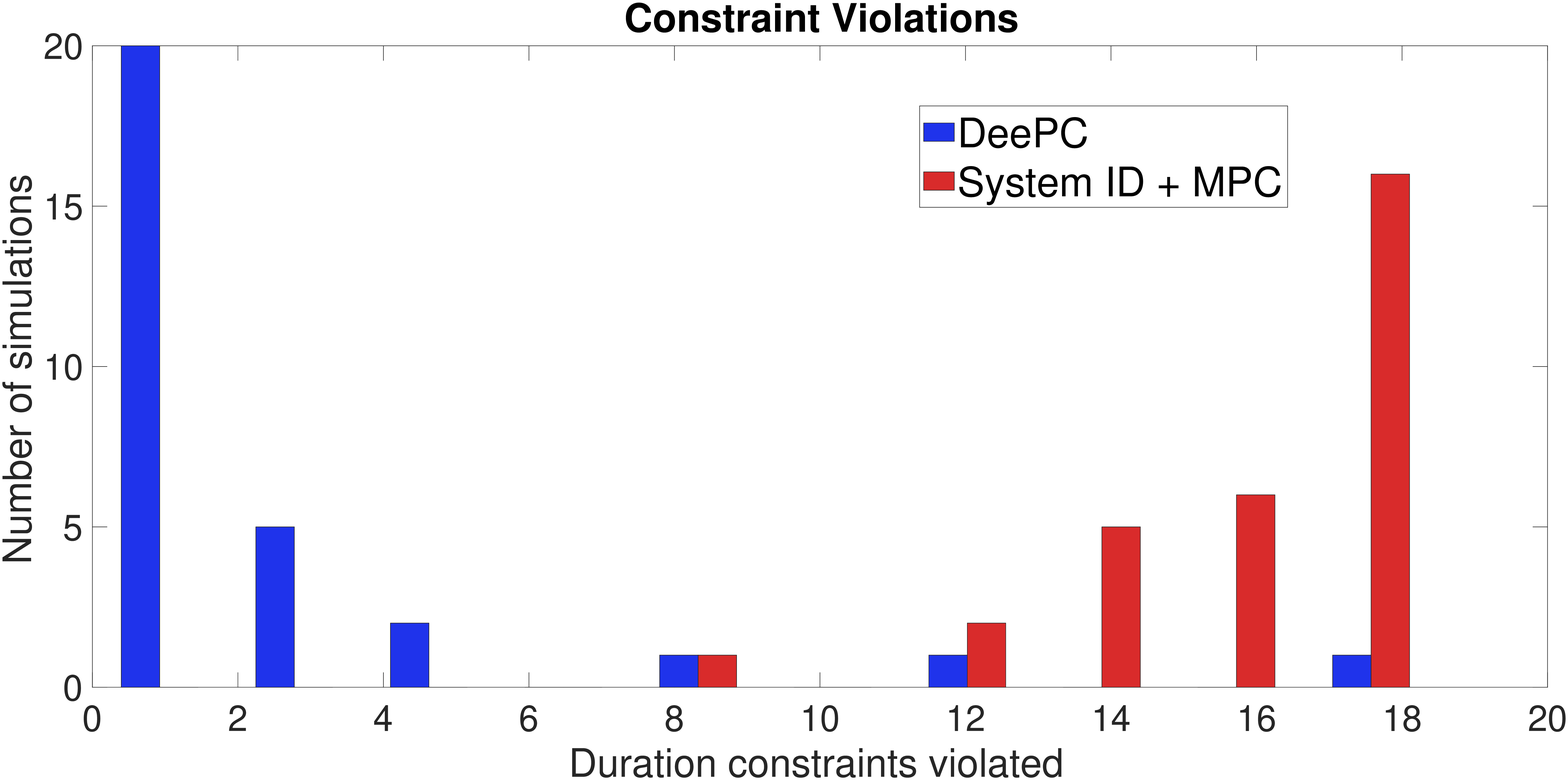} 
		\end{minipage}
		
		\begin{minipage}[h]{\linewidth}\vspace{3mm}
		\end{minipage}
		
		\begin{minipage}[h]{0.9\linewidth}
			\includegraphics[width=\linewidth]{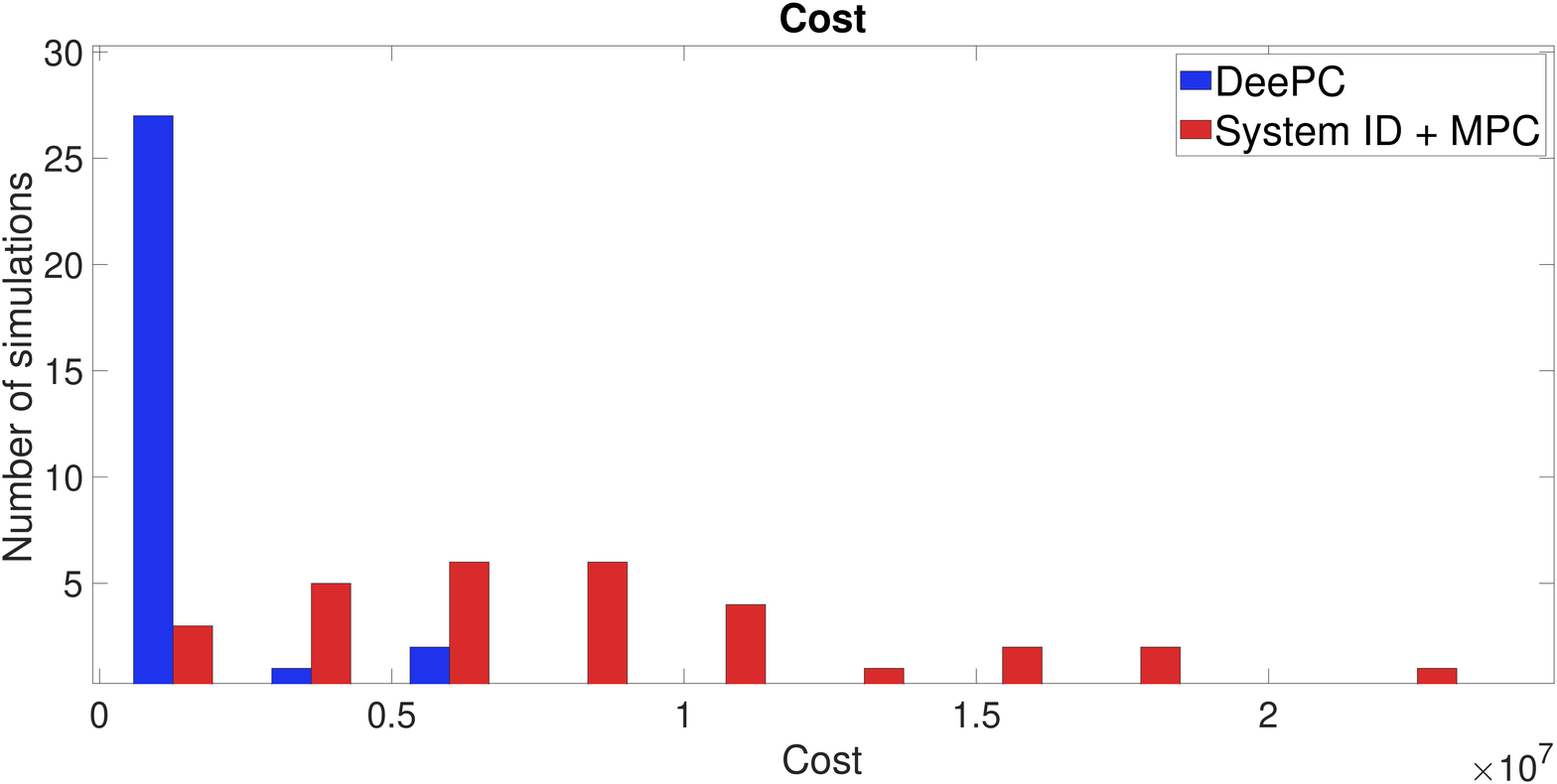} 
		\end{minipage}
	\caption{Constraint violations and cost comparison of DeePC and system ID with MPC.}
	\label{fig:hist}
\end{figure}

\begin{figure}[htb!]
	\centering
	\begin{minipage}[h]{0.48\linewidth}
		\includegraphics[width=\linewidth]{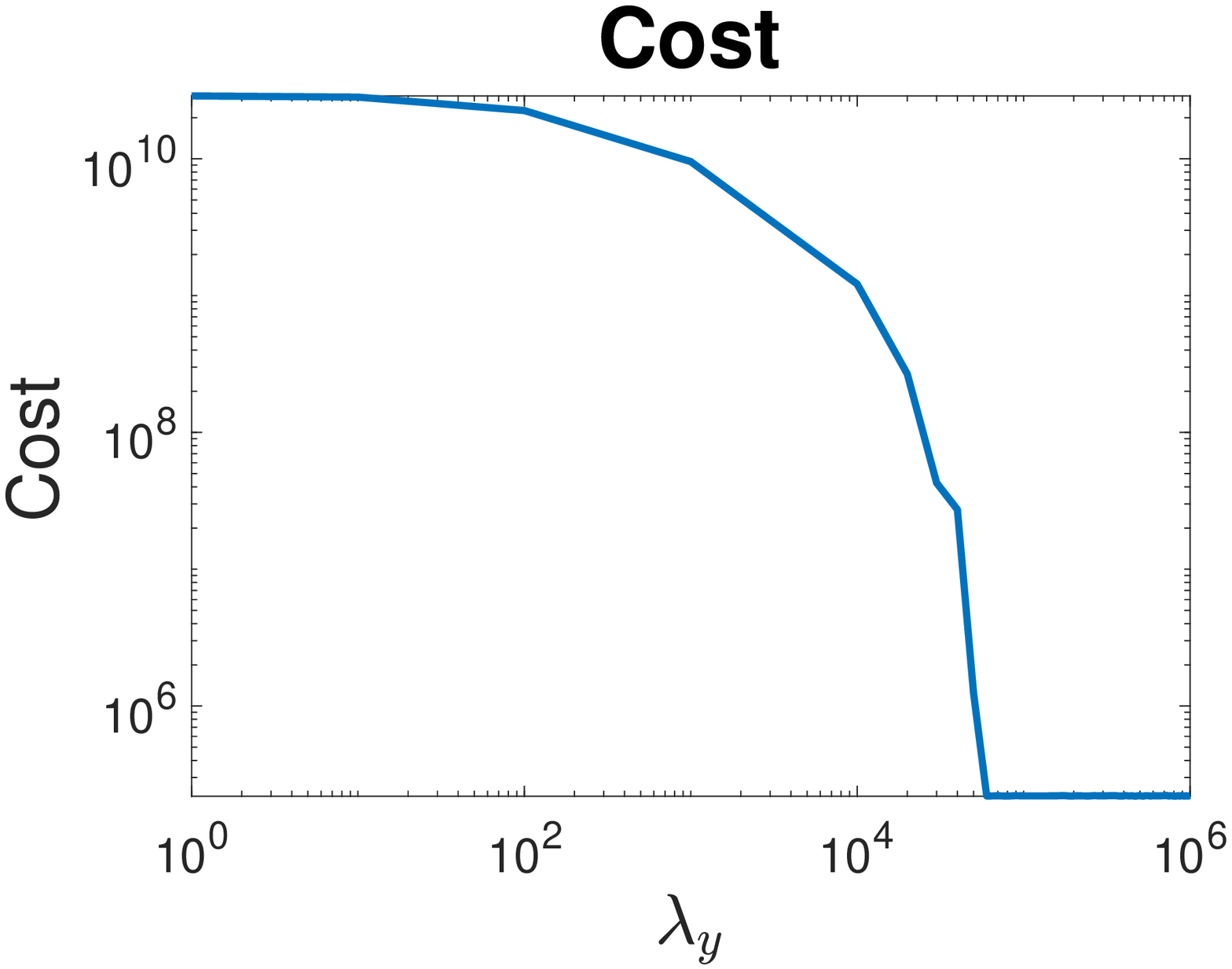}
		\subcaption{Fixed $\lambda_g = 300$}
		\label{fig:lambda_y_cost}
	\end{minipage}
	~
	\begin{minipage}[h]{0.48\linewidth}
		\includegraphics[width=\linewidth]{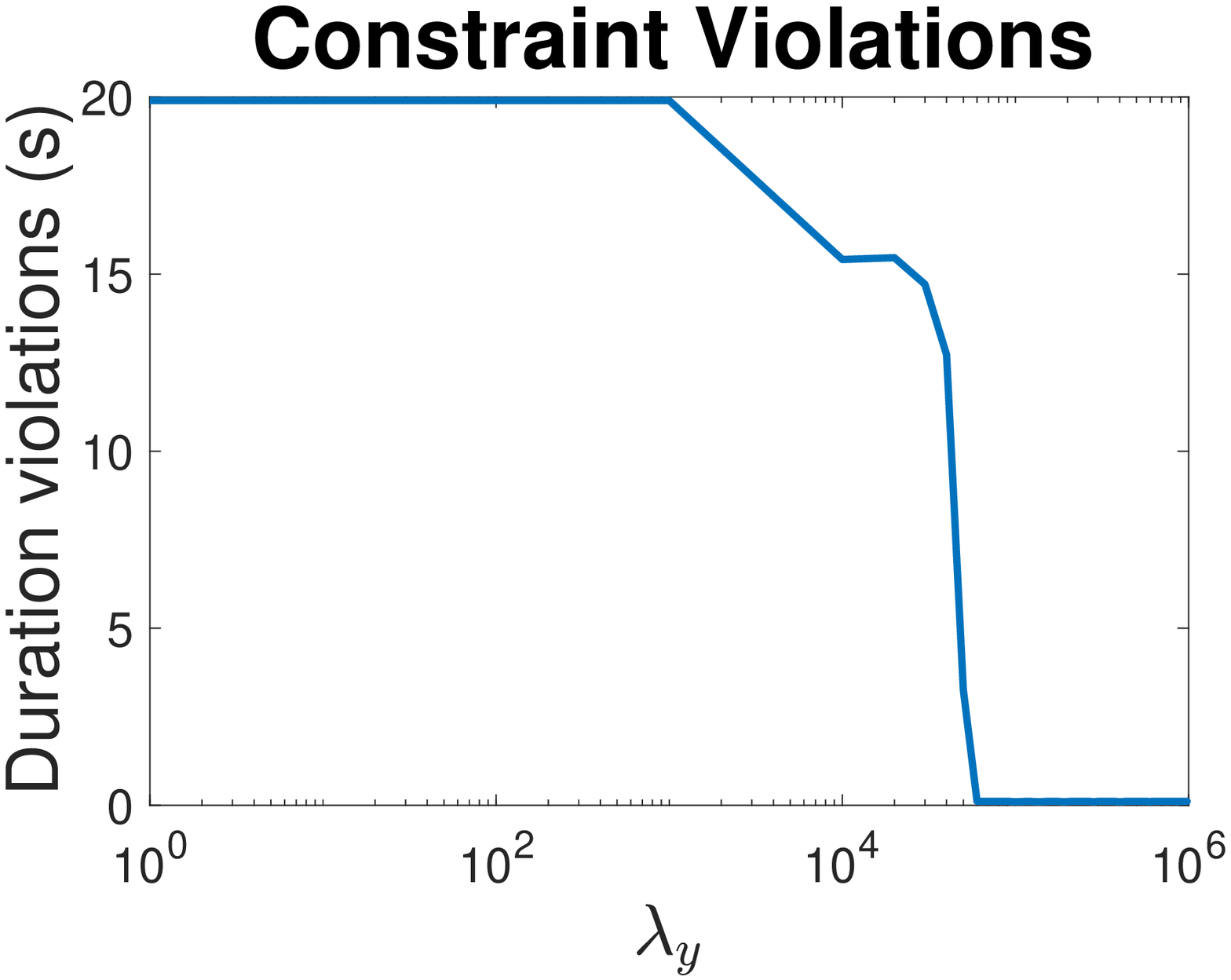}
		\subcaption{Fixed $\lambda_g = 300$}
		\label{fig:lambda_y_constraints}
	\end{minipage}
	
\begin{minipage}[h]{0.48\linewidth}\vspace{2mm}
\end{minipage}
~
\begin{minipage}[h]{0.48\linewidth}\vspace{2mm}
\end{minipage}

		\begin{minipage}[h]{0.48\linewidth}
		\includegraphics[width=\linewidth]{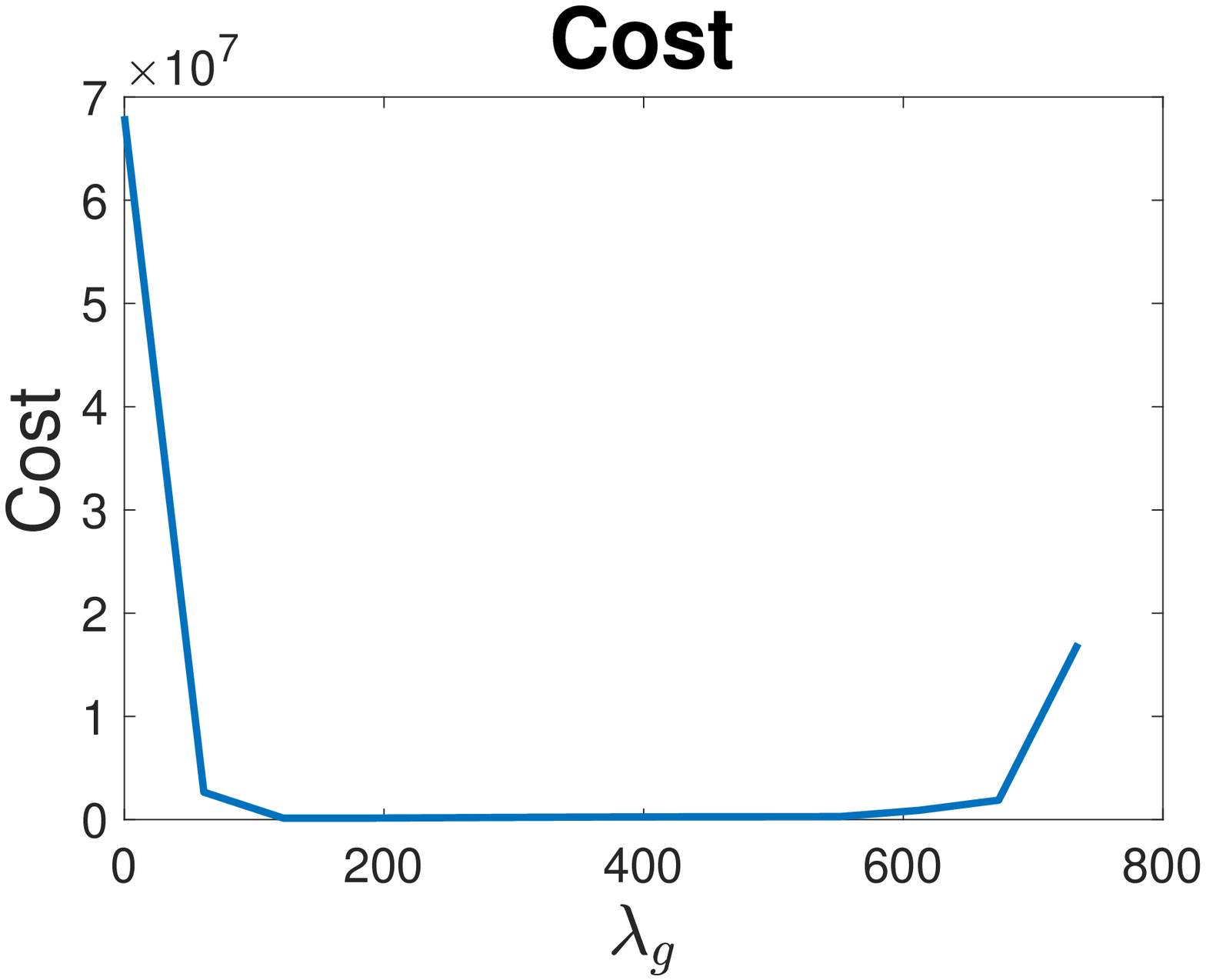}
		\subcaption{Fixed $\lambda_y = 10^5$}
		\label{fig:lambda_g_cost}
	\end{minipage}
	~
	\begin{minipage}[h]{0.48\linewidth}
		\includegraphics[width=\linewidth]{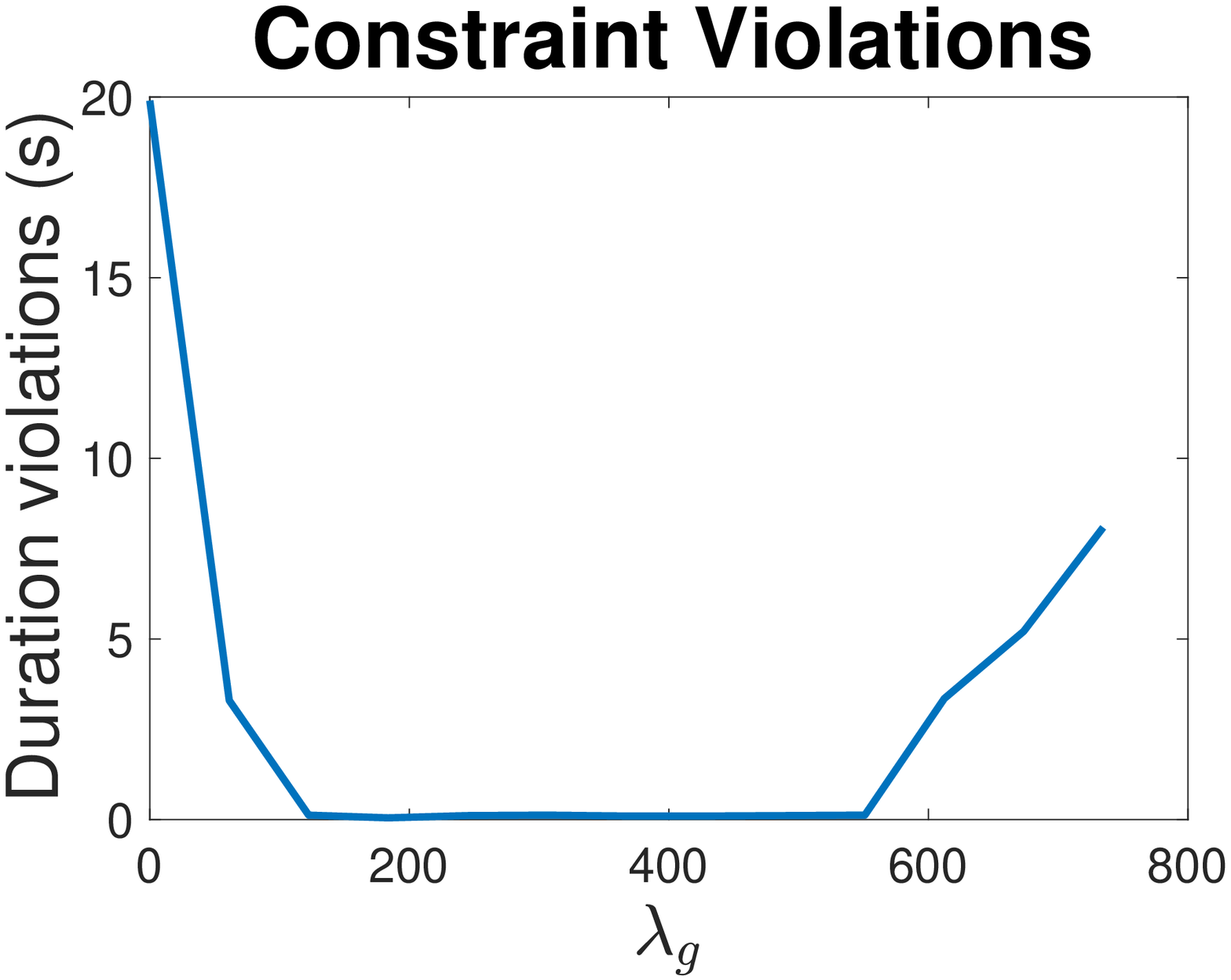}
		\subcaption{Fixed $\lambda_y = 10^5$}
		\label{fig:lambda_g_constraints}
	\end{minipage}
	\caption{Performance of quadcopter controlled by the DeePC algorithm with different regularization parameters.}
	\label{fig:regularizations}
\end{figure}

\section{Conclusion}\label{sec:conclusion}
We presented a data-enabled algorithm that can be applied to unknown LTI systems and formally showed its equivalence to the classical MPC algorithm. The DeePC algorithm uses a finite data set to learn the behaviour of the unknown system and computes optimal controls using real-time feedback to drive the system along a desired trajectory while respecting system constraints. Furthermore, we simulated a regularized version of the algorithm on stochastic nonlinear quadcopter dynamics illustrating its capabilities beyond deterministic LTI systems. The performance was superior when compared to system ID followed by MPC. Ongoing and future work focuses on the robustness of the DeePC algorithm and its regularization when applied to stochastic and nonlinear dynamics.
\section*{Acknowledgements}
We thank Bassam Bamieh, Ben Recht, Ashish Cherukuri, and Manfred Morari for useful discussions, and Melanie Zeilinger and Colin Jones for providing the simulation model.

\bibliographystyle{IEEEtran}%
\bibliography{IEEEabrv,JC}
\end{document}